\documentclass[1p]{elsarticle}

\usepackage[latin1]{inputenc}

\journal{Journal of \LaTeX\ Templates}

\usepackage{amssymb}
\usepackage{eucal}
\newcommand{\R}{{\Bbb R}}

\newcommand{\N}{{\Bbb N}}

\def \n  {\nonumber}

\def \h  {\hspace{1cm}}
\usepackage{color}

\usepackage{amsmath}

\newtheorem{thm}{Theorem}
\newtheorem{lemma}[thm]{Lemma}

\newtheorem{definition}[thm]{Definition}

\newproof{proof}{Proof}

\begin{document}

\begin{frontmatter}
%

\title{Nonstandard quasi-monotonicity: an application to the wave existence in a neutral KPP-Fisher equation
 }

\author[a]{ Eduardo Hern\'andez\,}
\author[b]{     Sergei Trofimchuk\footnote{Corresponding author.}}
\address[a]{Departamento de Computa\c{c}\~ao  e  Matem\'atica,  Faculdade de Filosofia,   Ci\^encias e Letras  de Ribeir\~ao Preto,  Universidade de S\~ao Paulo,  CEP 14040-901  Ribeir\~ao Preto, SP, Brazil
\\ {\rm E-mail: lalohm@ffclrp.usp.br}}
\address[b]{Instituto de Matem\'atica y F\'isica, Universidad de Talca, Casilla 747,
Talca, Chile \\ {\rm E-mail: trofimch@inst-mat.utalca.cl}}

\begin{abstract}
We revisit  Wu and Zou non-standard quasi-monotonicity  approach for proving existence of monotone  wavefronts in monostable  reaction-diffusion equations with delays.  This allows to  solve  the problem of existence  of monotone wavefronts in a neutral KPP-Fisher equation. In addition,  using some new ideas proposed recently by Solar {\it et al.},  we establish
 the uniqueness (up to a translation) of  these monotone wavefronts. \end{abstract}


\end{frontmatter}


\section{Revisiting  non-standard quasi-monotonicity approach to the wavefront existence problem}
In \cite{wzI,wz}, J. Wu and X. Zou developed a monotone iteration technique for proving the existence of monotone traveling  fronts to some classes of monostable functional differential equations with bounded delays. Later their method was extended  by Wang {\it et al}  \cite{WLR}  for the more general case of unbounded spatio-temporal delays.  A remarkable difference of the Wu-Zou work with other approaches is that their technique is constructive enough to obtain good analytical approximations of monotone  wavefronts \cite{GT,SZ}.
Obviously,  Wu and Zou iteration method can be applied only if the differential equation itself possesses monotone wavefronts.  As the studies \cite{GT,TPTJDE} show, often this happens  if and only if the nonlinearity of equation displays some kind of quasi-monotonicity.  Accordingly, \cite{wz}  considers two different types of the quasi-monotonicity property for the  monostable delayed model system
\begin{equation}\label{BE}
\partial_t u(t,x) = D\Delta u(t,x)  + f(u_t), \quad D = \mbox{diag}\{d_1,\dots,d_m\}, \ u \in \R^m_+,
\end{equation}
where $u_t$ stands for $u_t(\cdot,x) = u(t+\cdot,x) \in C([-\tau,0], \R_+^m)$ for each  fixed $x \in \R^m$ and the function $ f: C([-\tau,0], \R_+^m) \to \R^m$ is continuous and has exactly two zeros, $u_1\equiv 0$ and $u_2 = (\kappa_1, \dots, \kappa_m),$ $\kappa_j >0,$ in the rectangle $\mathcal{R}= [u_1,u_2]$ of the ordered space of constant functions $\R_+^m \subset C([-\tau,0], \R_+^m)$.

The first type is similar to (in fact, slightly stronger than) the quasi-monotonicity condition of  Martin-Smith \cite{MS}: basically, it  means that the nonlinearity $f$ of system (\ref{BE}) is monotone with respect to delayed terms.  Consequently, the second type of quasi-monotonicity proposed in \cite{wz} (see {\bf QM} below) is designed precisely to avoid this restrictive  standard monotonicity requirement on the delayed terms:

      {\begin{center}\begin{minipage}[t]{12cm}
    {\it {  ${ \bf QM}$ \it   Given $c >0$, there exists a non-negative matrix $B = \mbox{diag}\{\beta_1,\dots,\beta_m\}$,  such that $$f(\psi(-c\cdot)) - f(\phi(-c\cdot)) + B(\psi(0)-\phi(0)) \geq 0$$ for all
continuous functions $\phi, \psi \in C([-c\tau,0], \R_+^m)$ satisfying, in a componentwise way, the inequality   $0\leq \phi(s) \leq \psi(s) \leq u_2, \ s \in [-c\tau,0]$ and such that
$e^{Bt}(\psi(t) - \phi(t))$ is  componentwise nondecreasing on $[-c\tau,0]$.        }}
 \end{minipage}\end{center}

The nonstandard ordering of the space $C([-c\tau,0], \R_+^m)$ defined in two last lines of  {\bf QM}  was suggested by the previous work of  Smith and Thieme \cite{STa,STb}:
actually {\bf QM}  means monotonicity of the function $f$ with respect to this nonstandard order.  The existence of wavefronts under assumption {\bf QM}  was proved in \cite{HZ,LWT,WLR,WL,wz} for either delayed or non-local equations and also more recently for the neutral delayed systems  in \cite{HeT}.

To see the practical consequences of assuming  {\bf QM}, it is convenient
to consider the scalar case (i.e. to take $m=1$ and $u_2 = \kappa$: since the matrices $D, B$ are diagonal, this does not restrict the generality of our analysis) and focus  attention on a  functional of the typical  form  $f(\phi) = g(\phi(0), \phi(-\tau))$, where
 $g:\R^2_+ \to \R$ is  continuously differentiable  function.   Under these assumptions, {\bf QM} essentially reads as

  {\begin{center}\begin{minipage}[t]{12cm}
    {\it {  ${ \bf  Qm}$ \it  Given $c >0$, there exists $\beta > 0$
  such that $$\frak{D}:=g(A, B)- g(a, b) + \beta(A-a) \geq 0$$ for all
nonnegative $a\leq A\leq \kappa, b\leq B\leq \kappa$    such that $e^{-\beta c\tau}(B-b)  \leq (A-a)$.      }}
 \end{minipage}\end{center}

Simplifying, let us  assume additionally that $g(x,y)$ is nondecreasing in $x \in [0,\kappa]$ for each $y\in [0,\kappa]$, with
$\min \{\partial_xg(x,y),\ 0 \leq x, y \leq \kappa\} =0$ (this is true, in particular, for  the delayed KPP-Fisher equation
when $g(x,y) = x(1-y), \ \kappa = 1$).
Then  for all $a, A, b, B$ as in {\bf Qm} and some $\theta \in [b,B]$, it holds
$$\frak{D}\geq  g(A, B)- g(a, B)+ [g_2(a,\theta)+ \beta e^{-\beta c \tau}] (B-b) \geq 0$$
once
the partial derivative $g_2(x,y) := \partial_y g(x,y)$ satisfies
\begin{equation}\label{pin}
\beta e^{-\beta c \tau} \geq |g_-|, \ \mbox{where} \ g_-:  =\min \{g_2(x,y),\ 0 \leq x, y \leq \kappa\}.
\end{equation}
Under  the above described conditions,   inequality (\ref{pin}) is practically equivalent to {\bf Qm}  and it is satisfied
with the optimal $\beta = (c\tau)^{-1}$ whenever $c\tau \leq (|g_-|e)^{-1}$.  The latter  restriction on $c\tau$ indicates the range of
the applicability of the method.\footnote{As a consequence of restrictions imposed by nonstandard order on upper and lower solutions,  there are other additional bounds on $c\tau$:  however, they seem to be less important.} In the particular case of the
mentioned KPP-Fisher delayed equation (when $g_-=-1$), the above inequality becomes $c\tau \leq 1/e$ that is   far from conditions of the sharp criterion
for the existence of monotone fronts in the aforementioned equation \cite{GT,KO}. For example, this criterion shows that the only condition $\tau \leq 1/e$ is already sufficient for the existence
of monotone fronts propagating with speeds  $c \geq 2$.

The above analysis and computations suggest the following natural simplification of the hypothesis {\bf Qm} avoiding the use of the nonstandard
order in  the phase space:

 {\begin{center}\begin{minipage}[t]{12cm}
    {\it {  There exist $\alpha \geq 0, \beta > 0$
  such that $g(A, B)- g(a, b) + \alpha (A-a) + \beta(B-b) \geq 0,$ for all
nonnegative $a\leq A\leq \kappa,\  b\leq B\leq \kappa$.      }}
 \end{minipage}\end{center}

\vspace{1mm}

That is, instead of controlling monotonicity of delayed term indirectly (as in {\bf Qm},   by means of a non-standard order allowing additional relation between
$A-a$ and $B-b$), it is more natural to approach  the delayed term directly.  More generally, instead of  {\bf QM}  we can consider the following nonstandard quasi-monotonicity assumption on $f$:

 {\begin{center}\begin{minipage}[t]{12cm}
    {\it {  ${ \bf NS}$ \it     Given $c >0$, there is a  bounded linear operator  $L: C([-c\tau,0], \R_+^m) \to \R^m$   such that $$f(\psi(-c\cdot)) - f(\phi(-c\cdot)) + L(\psi-\phi) \geq 0$$ for all
  functions $\phi, \psi \in C([-c\tau,0], \R_+^m)$ satisfying, in a componentwise way, the inequality   $0\leq \phi(s) \leq \psi(s) \leq u_2, \ s \in [-c\tau,0]$.   }}
 \end{minipage}\end{center}

Remarkably, as we show in this paper, the above modification   of quasi-monotonicity conditions from \cite{wz} together with an adequate change of the respective  iterative algorithm  can  produce sharp criteria for the existence of monotone fronts.   Thus the  aim of this work is two-fold: first, we provide
an abstract result (fully in the sprit of \cite{wz}) assuring the existence of solutions to the wave profile equation whose nonlinearity meets nonstandard quasi-monotonicity condition {\bf NS} (see Theorem \ref{main1} below); secondly, we apply Theorem \ref{main1} to derive  a sharp criterion (see Theorem \ref{main2} below)  for the existence of monotone waves to the following neutral functional differential equation of the KPP-Fisher type
\begin{equation}\label{KFN}
\partial_t (u(t,x) - b u(t-\tau,x))= \partial_{xx} (u(t,x) - b u(t-\tau,x))  + u(t,x)(1-u(t-\tau,x)).
\end{equation}
It should be noted that for  functional differential reaction-diffusion equations of neutral type,  the wave propagation phenomena  are much less understood than in the case
of usual delayed or non-local reaction-diffusion equations. We are aware only of three recently published studies  \cite{HW,LW,YL} and  e-print \cite{HeT}.  Equation (\ref{KFN}) is
of the form  studied in  \cite{LW,YL}, however, the theory developed in the cited works cannot be applied to (\ref{KFN}) precisely  because of the lack of `standard' monotonicity  in the reaction term.
On the other hand, our extension of Wu and Zou  iteration technique can be efficiently applied to (\ref{KFN})   due to a non-standard quasi-monotonicity property  of the same reaction term.

In order to include  equations with unbounded spatio-temporal delays in the theory, instead of (\ref{BE}) it is convenient to work directly with  so-called wave profile system
\begin{equation}\label{PS}
D\phi''(t)  -c\phi'(t) + F(\phi)(t) =0, \quad \phi(-\infty)= 0, \ \phi(+\infty)= u_2,
\end{equation}
where $D = \mbox{diag}\{d_1,\dots,d_m\}, \ \phi \in \R^m_+,$ $c \in \R$ is some parameter (wave's speed) and $F: C_b(\R,\R^m)\to  C_b(\R,\R^m)$ is a nonlinear operator
defined on the space of all continuous bounded functions $\phi: \R \to \R^m$ equipped with the sup-norm. 
Assuming that $F$ is translation invariant (see  (c) in Theorem \ref{main1}),  by restricting $F$ on the constant functions $\phi(t,x) \equiv x$, we define  $\tilde F: \R^m\to \R^m$ as
$\tilde F(x) = F(\phi(t,x))$. Our basic assumption  consists in that $\tilde F$ has only two zero, $x=u_1 = 0$ and $x = u_2$, in the rectangle $\frak{R}=[0,u_2] \subset \R^m_+$.
For our purposes, it suffices to assume the following  weak continuity condition on $F$:
\begin{definition} We say that $F: C_b(\R,\R^m)\to  C_b(\R,\R^m)$ is a $p$-continuous map if for every bounded  sequence $\{\zeta_n\}$ of elements of
$C_b(\R,\R^m)$  pointwise converging to   $\zeta_* \in C_b(\R,\R^m)$, the sequence $\{F(\zeta_n)(t)\}$ converges to $F(\zeta_*)(t)$ at each $t\in \R$.

\end{definition}

As usual, we will need upper and lower  solutions of (\ref{PS}) as the initial approximations to the wavefront: importantly,  our definition of sub- and super-solutions does not differ from that given in \cite{wz,wzE}:
\begin{definition} \label{D1}  Let $t_1, t_2$ be some real numbers. We will call a super-solution of (\ref{PS}) each  continuous  and $C_b^2$-smooth on $\R\setminus\{t_1\}$,  componentwise positive  nondecreasing function  $\phi_+:\R \to \R^m_+,\  \phi_+(+\infty)= u_2,  \  \phi_+(-\infty)= 0,$  satisfying the differential inequalities
$$
D\phi_+''(t)  -c\phi'_+(t) + F(\phi_+)(t) \leq 0, \  t\not= t_1, \quad \phi'_+(t_1-) \geq  \phi'_+(t_1+).
$$
Similarly, a sub-solution of (\ref{PS})   is a non-negative  componentwise non-constant continuous function  $\phi_-:\R \to \R^m_+,\  \phi_-(-\infty)= 0, $ satisfying the inequalities
$$
D\phi_-''(t)  -c\phi'_-(t) + F(\phi_-)(t) \geq 0, \  t\not= t_2, \quad \phi'_-(t_2-) \leq  \phi'_-(t_2+).
$$
\end{definition}
%
%
Then we have the following simple general  existence result:\footnote{
A natural question is whether the uniqueness of wavefronts can be proved within the framework (or  under assumptions) of Theorem \ref{main1}. We believe that in general  this question has a negative answer. }
\begin{thm}\label{main1} Suppose that there exists  linear $p$-continuous operator $L: C_b(\R,\R^m)\to  C_b(\R,\R^m)$ such that the differential operator
$\mathcal{D}: C^2_b(\R,\R^m)\to  C_b(\R,\R^m)$ defined by
\begin{equation} \label{DO}
\mathcal{D}\phi(t) =    D\phi''(t)  -c\phi'(t) - (L\phi)(t), \ t \in \R,
\end{equation}
on the space $C^2_b(\R,\R^m)$ of twice continuously differentiable bounded, with bounded derivatives, functions, has the inverse integral operator
\begin{equation}\label{IO}
\mathcal{I}f(t) = \mathcal{D}^{-1}f(t) = \int_\R N(t-s)f(s)ds ,
\end{equation}
with non-positive continuous matrix-valued kernel $N:  \R \to M_m(\R_-)$. If, in addition,
 \begin{enumerate}
  \item[\bf(a)]  operator $F+L$ preserves  natural componentwise  order of
 $C_b(\R,\R^m)$,
 $$
 F(\psi) - F(\phi) +L(\psi-\phi) \geq 0, \ \mbox{for all} \ \psi \geq \phi, \ \psi, \phi \in C_b(\R,\R^m);
 $$
  \item[\bf(b)]  there is a pair of sub- and super-solutions $\phi_-(t) \leq \phi_+(t)$ for (\ref{PS});
  \item[\bf(c)]    $L$ and $F$ commute with the translation operator $(T_s\phi)(t):= \phi(t+s)$ for all  $s\in \R$,
\end{enumerate}
 then  (\ref{PS}) has a monotone solution $\phi_*(t)$ coinciding with   pointwise limit of the decreasing sequence of non-decreasing functions $\phi_{n+1} =
-\mathcal{I}[F(\phi_n)+L\phi_n], \ n =0,1, \dots; \ \phi_0:= \phi_+$.
\end{thm}
\begin{proof} Since each  solution $\phi(t)$ of the boundary value problem (\ref{PS}) is bounded, the functions $F(\phi)(t), (L\phi)(t) $ also are bounded on $\R$. Therefore
bounded continuous function $\phi$ satisfies the equation $\phi = A(\phi), $ $A(\phi):= -\mathcal{I}[F(\phi)+L\phi]$  if and only if it satisfies (\ref{PS}).
Clearly, in view of assumption (a) and non-negativity of $N$, we have that $A(\phi)(t) \leq A(\psi)(t)$ for each ordered par of functions $\phi, \psi \in C_b(\R,\R^m)$, $\phi(t) \leq \psi(t), \ t \in \R$. Since $\phi(t) \leq \phi(t+s)$ for each non-decreasing function $\phi$ and $s >0,\  t \in \R$, the monotonicity of $A$ together with (c)  imply that $A$ transforms non-decreasing functions in non-decreasing ones. Moreover, by Lemma \ref{L7} proved in Appendix, $\phi_- \leq A\phi_- \leq A\phi_+ \leq \phi_+$. Hence, the sequence $\phi_n$ defined in the statement of Theorem \ref{main1} and the sequence $\psi_n = F(\phi_n)+L\phi_n$,
are decreasing and pointwise converge to non-negative monotone bounded limit functions $\phi_*, \psi_*$.  From the definition of $N$ we have that all components of the
matrix $N(s)$ are non-positive and integrable on $\R$, in fact
$L\int_\R N(s){\bf 1 }ds= -{\bf 1}$ where ${\bf 1} = (1,\dots,1)$. Since $\phi_{n+1}(t) = - \int_\R N(t-s)\psi_n(s)ds$, we can use the Lebesgue's dominated convergence theorem to conclude that $\phi_*(t) = - \int_\R N(t-s)\psi_*(s)ds$. This implies that the function $\phi_*$ is continuous and that $\phi_n(t)$ converges to $\phi_*(t)$ uniformly on bounded subsets of $\R$ (this is because of the Dini's monotone convergence theorem). Consequently, since $F$ and $L$ are $p$-continuous, we obtain that $\phi_* =
-\mathcal{I}[F(\phi_*)+L\phi_*]$, so that non-decreasing function $\phi_*$ solves the differential equation of (\ref{PS}) and satisfies $\phi_-(t) \leq \phi_*(t)\leq \phi_+(t), \ t \in \R$.  Thus
$\phi_*(-\infty)=0$ and the vector $\Phi:= \phi_*(+\infty)\leq u_2$ has positive components. Finally, observe that the sequence   of functions $\{\xi_n(t) := \phi_*(t+n)\}$, $n \in \N$, is non-decreasing and converging to a constant function $\phi_*(+\infty)$.  The  assumption (c) and $p$-continuity of $L$, $F$ implies that $F(\xi_n)(t)+(L\xi_n)(t) = F(\phi)(t+n)+(L\phi)(t+n)$ pointwise converges to  the constant function $F\Phi+L\Phi$. Thus,  after taking limit  in $\xi_n=-\mathcal{I}[F(\xi_n)+L\xi_n]$ as $n \to +\infty$, we may conclude that $\Phi=\phi_*(+\infty)$ is a constant solution of (\ref{PS}). This means that $\tilde F(\phi_*(+\infty))=0$ and, consequently, $\phi_*(+\infty) = u_2$. \hfill $\square$
\end{proof}

As an application,  we are going to apply Theorem \ref{main1} to prove  the existence of monotone traveling wave $u(t,x) = \phi(x+ct)$ of equation  (\ref{KFN}), i.e. solutions such that the wave profile $  \phi(t)$ is  monotone on $\R$, the difference $v(t):= \phi(t)-b\phi(t-c\tau)$ is $C^2$-smooth on $\R$  and satisfies the following neutral functional differential equation
\begin{equation}\label{PE}
v''(t) -cv'(t) +  \phi(t)(1 - \phi(t-c\tau)) =0, \ t \in \R, \ \ \phi(-\infty)=0,\  \phi(+\infty)=1.
\end{equation}
As a consequence of the above definition, $\phi(t)$ is strictly positive on $\R$. Moreover, $c\not=0$ since otherwise  (\ref{PE}) reduces to a simple second order differential   equation
which does not admit monotone  solutions satisfying boundary conditions of (\ref{PE}). Thus, without loss of generality, we can assume that $c >0$ and
$$
v'(t_k) =  c^{-1}\phi(t_k)(1 - \phi(t_k-c\tau))
$$
 at every critical point $t_k$ of $v'(t)$. Since $v(\pm\infty)$ are finite, this implies that $v'(t)$ is a bounded function; then, by (\ref{PE}), $v''(t)$ is also bounded and $v'(t)$ is uniformly continuous on the real line. By a well know argument (Barbalat's lemma),  we conclude that $v'(\pm\infty) =0$.  In this way, after integrating (\ref{PE}) on $\R$, we obtain
 $$
 c(1-b) = \int_\R\phi(s)(1 - \phi(s-c\tau))ds >0.
 $$
 Therefore $c(1-b)>0,\ c >0, \ b >0$ yield that $b \in (0,1)$: similarly to \cite{LW}, this condition on $b$ will be assumed  everywhere in this paper.

Now, the  following characteristic functions will play a key role in  the statement of our second theorem:
$$
\chi_0(z) = z^2 - cz +\frac{1}{1-be^{-zc\tau}},\quad  \chi_1(z) =  z^2 - cz -\frac{e^{-zc\tau}}{1-be^{-zc\tau}}.
$$
It is easy to see that these functions are analytic in the half-plane $\Re z > \ln b/(c\tau)$ where $|1-be^{-zc\tau}| > 1-be^{-\Re z c\tau}$. The latter estimation also shows that every
half-plane $\Re z > a $ with $a > \ln b/(c\tau)$ contains a finite number of zeros of $\chi_j(z)$.
\begin{lemma} \label{L4} Assume that $b \in (0,1),\ \tau \geq 0$. Then $({\bf A})$ there exists a strictly decreasing continuous function $c_*=c_*(\tau) >2, \tau \geq 0,$ such that
$
c_*(0)=2(1-b)^{-1/2}, \  c_*(+\infty) =2
$
and $\chi_0(z)$ has exactly two positive zeros $\lambda_2(c) \leq \lambda_1(c)$ (counting multiplicity) if and only if $c \geq c_*(\tau)$. In fact, these zeros coincide only if $c=c_*(\tau)$.  $({\bf B})$  Furthermore, if $\lambda_j$ is a complex zero of $\chi_0(z)$ with $c \geq c_*(\tau),$ then $\Re \lambda_j < \lambda_2(c)$.
\end{lemma}
\begin{proof} After introducing the change of variable $\lambda= cz, \ \epsilon = c^{-2}$ and using the convexity properties of functions $-\epsilon \lambda^2 + \lambda$ and
$(1-be^{-\lambda \tau})^{-1}$, the statement of   $({\bf A})$  is immediate (see also (\ref{D1}) in Appendix).  The proof of  $({\bf B})$ is based on  standard  arguments  from complex analysis (e.g. see \cite[Appendix]{SoT}) and  it is omitted. \hfill $\square$
\end{proof}

The next lemma can be seen as an extension of \cite[Lemma 3]{GT} for $b \in (0,1)$:
\begin{lemma} \label{L5} Assume that $b \in (0,1)$ is a fixed number, $c, \tau > 0$. Then there exist finite $\tau(b)>0$ and a strictly decreasing continuous function $c_\#=c_\#(\tau) >0,\ \tau >\tau(b),$ satisfying
$
c_\#(\tau(b)-)=+\infty, \  c_\#(+\infty) =0
$
and such that  $\chi_1(z)$ has exactly two negative zeros $\mu_2(c) \leq \mu_1(c)$ (counting multiplicity) if and only if
either $\tau \in (0,\tau(b)]$ or $\tau > \tau(b)$ and
$c \leq c_\#(\tau)$. These zeros coincide only if $c=c_\#(\tau)$.  Function $\tau(b)$ is strictly decreasing on $(0,1)$, with $\tau(0+) = 1/e$ and $\tau(1-)= 0$.
\end{lemma}
\begin{proof}  See Lemma \ref{LL} in Appendix.   \hfill $\square$
\end{proof}

Lemmas \ref{L4} and \ref{L5} show that the following closed subset of $\R_+^2$: $$\frak{D} = \{(\tau, c)\in \R_+^2:\ \chi_0(z) \ \mbox{has positive zeros and }   \chi_1(z) \ \mbox{has negative zeros}\}$$
is non-empty (actually, it is unbounded and simply connected, since the curves $c=c_*(\tau), \ c=c_\#(\tau)$ intersect transversally exactly at one point, see Lemma \ref{IT} in Appendix). Here is our second theorem:
\begin{thm}\label{main2} Suppose that $b\in [0,1), \tau >0$. Then equation (\ref{KFN}) has a monotone traveling front $u(t,x) = \phi(x+ct)$ if and only if $(\tau,c) \in \frak D$.
Moreover, profile $\phi$ of this solution can be approximated by successive iterations as described in Theorem \ref{main1}.
\end{thm}

It is interesting to note that equation (\ref{KFN})  with $b=0$ coincides with the usual KPP-Fisher delayed equation, cf.  \cite{BS,BNPR,ADN,FZ,GT,LMS,KO,SoT,wz}.  There are at least four distinct demonstrations  of the wavefront existence criterion  for this equation (see \cite{FZ,GT,LMS,KO}): our proof here differs from the previously known ones. Now, the form of dependence of $\frak{D}$ on $b\geq 0$ shows that  the neutral correction $-bu(t-c\tau)$ in (\ref{PE})  increases the minimal speed of propagation and has a negative impact on the monotonicity properties of its wave solutions.

An important additional question is whether traveling waves to the neutral diffusive equations have the same kind of uniqueness and stability properties as the  delayed diffusive equations.
For general neutral systems of  \cite{LW}  which satisfy standard quasi-monotonicity condition, the uniqueness of all {\it non-critical} waves was proved in \cite{YL}.  Since equation (\ref{PE}) is not quasi-monotone in the usual sense,  the theory of \cite{YL} does not apply to it. After rewriting equation (\ref{PE}) in terms of function $v(t)$, we obtain an equivalent  non-neutral equation (\ref{PEW}):  this suggests that the approach to the uniqueness problem developed in \cite{FZ,GT} can be useful in our situation. This approach, however, can be applied only to monotone waves and it also needs exact a priori estimates of wave   asymptotics   at $+\infty$ (which seems to be  more difficult to justify than similar estimates at $-\infty$).  Therefore, in this work,  we are invoking a new idea recently proposed in \cite{SoT} for functional diffusive equations with finite delays.
The next theorem and its proof show that the mentioned idea in   \cite{SoT}  is also significant  for some equations with infinite delays.
\begin{thm}\label{main3} Suppose that $b\in [0,1)$.  Then for every fixed pair  $(\tau,c) \in \frak D$  equation (\ref{KFN}) has a unique (up to translation) monotone traveling front $u(t,x) = \phi(x+ct)$.
\end{thm}

  Observe that,  in difference with \cite{YL}, Theorem \ref{main3} includes critical waves. On other hand, it establishes uniqueness property within the class of monotone waves only (as in \cite{FZ,GT,YL,TPTJDE}).  The analysis  of \cite{SoT} suggests that  Theorem \ref{main3}  can be improved to include more general {\it semi-wavefronts}.  We do not pursue this goal in the present work since it  requires
more exhaustive analysis of asymptotic behavior of waves at $-\infty$.

Finally, a few words about the organization of the paper. We prove  Theorem \ref{main2} in Section 2.   The proof of necessity of condition $(\tau,c) \in \frak D$ in this theorem demands asymptotic analysis of wave profiles at $-\infty$. This analysis is realized in Lemma \ref{sms2} which is then  used in  Section 3 where  Theorem \ref{main3} is proved.  Section 3 also contains a general result (Lemma \ref{sms3}) about non-existence
of  super-exponentially decaying (at $-\infty$) solutions to some asymptotically autonomous differential equations with infinite delay.
Finally, Appendix contains proofs of several technical assertions used along the paper.
\section{Proof of Theorem \ref{main2}}

\subsection{Sufficiency}

   Consider   bounded linear operator $Sg(t):=g(t-c\tau)$ acting in the space  $C_b(\R,\R^m)$ provided with the sup-norm. Then we have that  $|S|=1$  and the relation
 $w(t):= (\phi(t)-b\phi(t-c\tau))/(1-b)$ can be written as $(1-b)w = (I-bS)\phi$ where $I$ is the identity operator.  Since $b\in (0,1)$, we obtain that
 $$\phi = (1-b)Bw:= (1-b)(I-bS)^{-1}w = (1-b)\sum_{j=0}^{+\infty}b^jS^jv \in C_b(\R,\R^m). $$
 Thus, after multiplying equation (\ref{PE}) by $1/(1-b)$, we get the following equivalent problem
 \begin{equation}\label{PEW}
w''(t) -cw'(t) + (Bw)(t)(1 - (1-b)(SBw)(t)) =0, \ t \in \R, \ \ w(-\infty)=0, w(+\infty)=1.
\end{equation}
Consider operators $L= SB, \ F: C_b(\R,\R^m)\to  C_b(\R,\R^m)$ where
$$
F(\phi)(t) :=  (B\phi)(t) - (1-b)(B\phi)(t)\cdot (SB)\phi(t).
$$
Note that the operators  $S, B = (1-bS)^{-1}, L = SB $  and $ F$ are clearly $p$-continuous and commute with the translation operator $T_s, \ s \in \R$.

Then, since the operators $S, B$ preserves the natural order of $C_b(\R,\R^m)$, for all $\phi, \psi \in C_b(\R,\R^m)$ satisfying $0 \leq \phi(t) \leq \psi(t) \leq 1, \ t \in \R$, it holds that
  \begin{eqnarray}
{\lefteqn{(L(\psi-\phi))(t) + F(\psi)(t) - F(\phi)(t) }}\h\h \n  \\
&=&  \left((B\psi)(t) - (B\phi)(t)\right)(1-  (1-b)(SB\psi)(t)) + \n\\
&&    +\left[(SB\psi)(t)- (SB\phi)(t)\right](1- (1-b)(B\phi)(t))\geq 0, \ t \in \R.\n
\end{eqnarray}
Notice that  $(1-b)(SB\psi)(t),  (1-b)(B\phi)(t) \in [0,1] $ for all $t \in \R$.

\vspace{2mm}

To apply Theorem \ref{main1}, we will need the next auxiliary results, Lemmas \ref{LV}, \ref{US} and \ref{LS}.

\begin{lemma}\label{LV}  Assume that $(\tau,c)\in \frak{D}$. Then the differential operator
$\mathcal{D}: C^2_b(\R,\R)\to  C_b(\R,\R)$ defined by formula (\ref{DO}) (where $D=1$ should be taken)
has the inverse integral operator in the form (\ref{IO}) with negative continuous matrix-valued kernel $N:  \R \to (-\infty,0)$.
\end{lemma}
\begin{proof} Observe that, using the Dirac $\delta$-function $\delta(t)$, we can represent $(SB)\phi(t)$ as
$$
(SB)\phi(t) = \sum_{j=0}^{+\infty}b^j\phi(t-(j+1)c\tau) = \int_\R K(t-s)\phi(s-c\tau)ds,  \quad K(t):=  \sum_{j=0}^{+\infty}b^j\delta(t-j\tau) .
$$
Then we observe that the theory developed in \cite{TPTJDE} for the Lebesgue integrable kernel $K(s)$ applies literally to the case of the generalized kernels.
In particular,  in view of Lemma \ref{L5} above, \cite[Lemma 19]{TPTJDE} holds with $\xi_*\geq 1$. Observe that we can admit  $d=0$ in \cite[Lemma 19]{TPTJDE}
because of the following property of the above given kernel $K(s)$: if $v \in C_b(\R,\R)$ satisfies $v(t) <0$ for all $t \leq a_c$ then  $K * v(a_c-c\tau) = \sum_{j=0}^{+\infty}b^jv(a_c -{(j+1)}\tau) <0$, cf. \cite[Claim IV, p. 1221]{TPTJDE}. \hfill $\square$
\end{proof}

For $(c,\tau) \in \frak{D}$ with $c> c_*(\tau)$,  we will consider the following function
$$ \phi_+(t) = \left\{
\begin{array}{ccc}
1- e^{\mu_1 t} &,& t \geq \zeta \\
a e^{\lambda_2 t}&,& t \leq \zeta
\end{array}, \right. $$
where positive $a$ and $\zeta \in \R$ are chosen to assure the continuity of the derivative $\phi_+'(t)$ on $\R$. Due to the opposite convexities of the pieces of $\phi_+$,
the existence of such $a, \zeta$ is immediate; furthermore, by the same reason,
\begin{equation}\label{aux}
1- e^{\mu_1 t} < a e^{\lambda_2 t}, \quad t < \zeta.
\end{equation}
This choice of function  $\phi_+$ was suggested by the studies in \cite{FZ,GT}.
\begin{lemma}\label{US}  For $(c,\tau) \in \frak{D}$ with $c> c_*(\tau)$, $\phi_+(t)$ is a super-solution for  (\ref{PEW}).
\end{lemma}
\begin{proof} We only need to check the inequality
\begin{equation}\label{In}
\phi_+''(t) -c\phi_+'(t) + (B\phi_+)(t)(1 - (1-b)(SB\phi_+)(t)) \leq 0, \quad t\not=\zeta.
\end{equation}
Since $e^{\lambda_2 t}$ is an eigenfunction for the linear equation $y''(t) -cy'(t) + (By)(t)=0$, relation (\ref{In}) clearly holds for all $t < \zeta$.

Now, if $t \in [\zeta + mc\tau, \zeta + (m+1)c\tau]$ with $m \geq 0$, using (\ref{aux}) we find  that
\begin{eqnarray}
{\lefteqn{ 1 - (1-b)(BS\phi_+)(t) }}\h \n\\
 &= & 1 - (1-b)\sum_{j=0}^{+\infty}b^j\phi_+(t-(j+1)c\tau) \n\\
&=& 1 - (1-b) [\sum_{j=0}^{m-1}b^j(1-e^{\mu_1(t-c\tau(j+1))})+ \sum_{j=m}^{+\infty}b^jae^{\lambda_2(t-c\tau(j+1))} ]\n\\
&=&  1 - (1-b)\sum_{j=0}^{+\infty}b^j(1-e^{\mu_1(t-c\tau(j+1))})\n\\
&& + (1-b)\sum_{j=m}^{+\infty}b^j(1-e^{\mu_1(t-c\tau(j+1))} -ae^{\lambda_2(t-c\tau(j+1))})\n\\
&<&   (1-b)\sum_{j=0}^{+\infty}b^je^{\mu_1(t-c\tau(j+1))} = (1-b)e^{\mu_1t}\frac{e^{-\mu_1c\tau}}{1-be^{-\mu_1c\tau}}.\n
\end{eqnarray}
Therefore, for the same values of $t$, since $B\phi_+ < B 1 = 1/(1-b),$ we obtain that
\begin{eqnarray}
{\lefteqn{\phi_+''(t) -c\phi_+'(t) + (B\phi_+)(t)(1 - (1-b)(SB\phi_+)(t)}}\n \h\h  \\
&<& (-\mu_1^2+c\mu_1)e^{\mu_1t} + (1-b)e^{\mu_1t}\frac{e^{-\mu_1c\tau}}{1-be^{-\mu_1c\tau}} (B\phi_+)(t)\n\\
&=&
-e^{\mu_1t}\frac{e^{-\mu_1c\tau}}{1-be^{-\mu_1c\tau}} e^{\mu_1t}\left[1 - (1-b)(B\phi_+)(t)\right] <0. \nonumber
\end{eqnarray}
This completes the proof of Lemma \ref{US}.  \hfill $\square$
\end{proof}

Next, for $(c,\tau) \in \frak{D}$ with $c> c_*(\tau)$,  $M>1,\epsilon >0$ and for $a$ as in $\phi_+(t)$, we will consider the following well known ansatz for  sub-solution
$$ \phi_-(t) = \left\{
\begin{array}{ccc}
ae^{\lambda_2 t}(1- Me^{\epsilon t}), && t \geq \xi, \\
0, && t \leq \xi.
\end{array} \right. $$
Here $\xi=\xi(M,\epsilon)$ is chosen to assure the continuity of  $\phi_-(t), \ t \in \R$. Clearly, $\phi_-(\xi-) < \phi_-(\xi+)=0$ and $\phi_-(t) < \phi_+(t)$ for all $t \in \R$.
\begin{lemma}\label{LS}  For $(c,\tau) \in \frak{D}$ with $c> c_*(\tau)$, and some appropriate  $M>1,\ \epsilon \in (0,\lambda_2)$, the function $\phi_-(t)$ is a sub-solution for  (\ref{PEW}).
\end{lemma}
\begin{proof} In view of the above said, it suffices to check the inequality
$$
\Psi:= \phi_-''(t) -c\phi_-'(t) + (B\phi_-)(t)(1 - (1-b)(SB\phi_-)(t) \geq 0
$$
only for $t < \xi$.  Since,  for these values of $t$,
$$
(B\phi_-)(t)= \frac{ae^{\lambda_2 t}}{1-be^{-\lambda_2c\tau}} - M\frac{ae^{(\lambda_2+\epsilon) t}}{1-be^{-(\lambda_2+\epsilon)c\tau}} < \frac{ae^{\lambda_2 t}}{1-be^{-\lambda_2c\tau}},
$$
and, consequently,
$$
(SB\phi_-)(t) < S\left( \frac{ae^{\lambda_2 \cdot}}{1-be^{-\lambda_2c\tau}}\right)(t)=  \frac{ae^{\lambda_2 (t-c\tau)}}{1-be^{-\lambda_2c\tau}},
$$
we conclude that for all $t < \xi$ and sufficiently large $M$,
\begin{eqnarray}
\Psi&=& -aM\chi_0(\lambda_2+\epsilon)e^{(\lambda_2+\epsilon)t} - (1-b) (B\phi_-)(t) (SB\phi_-)(t)\h \n\\
&\geq&
aMe^{(\lambda_2+\epsilon)t}\left[-\chi_0(\lambda_2+\epsilon)- a\frac {1-b} M \frac{e^{(\lambda_2-\epsilon) \xi}e^{-\lambda_2c\tau}}{(1-be^{-\lambda_2c\tau})^2}\right]\n\\
&=& aMe^{(\lambda_2+\epsilon)t}\left[-\chi_0(\lambda_2+\epsilon)-   \frac{a}{M^{^{\frac{\lambda_2}{\epsilon}}}} \frac{(1-b)e^{-\lambda_2c\tau}}{(1-be^{-\lambda_2c\tau})^2}\right] >0.\nonumber 
 \end{eqnarray}
The proof of Lemma \ref{LS} is completed.   \hfill $\square$
\end{proof}
Hence, the above lemmas show that for each $(c,\tau) \in \frak{D}$ with $c> c_*(\tau)$,  all assumptions of Theorem \ref{main1} are satisfied, implying the existence of a monotone wavefront
for equation (\ref{KFN}).  Now, in the critical case when $(c_0,\tau_0) \in \frak{D}$ and  $c_0=c_*(\tau_0)$, using simple geometry of the domain $\frak{D}$ (cf. Lemma \ref{IT}),
we can find a sequence of points $(c_n,\tau_n) \in \frak{D}$ with $c_n> c_*(\tau_n) >2$ and such that $c_n\to c_0, \ \tau_n \to \tau_0$.   By the previous conclusion, we know that for each  pair $(\tau_n, c_n)$ there is a monotone traveling wave $\phi_n(t)$ for equation (\ref{PEW}). Due to translation invariance property of $F$, we can
normalize $\phi_n$   by the condition $\phi_n(0)=0.5$. Now, applying the arguments given below equation (\ref{PE}) to the equation (\ref{PEW}), we conclude that the derivatives $\phi_n'(t)$ are uniformly bounded in $t\in \R$ and $n\in \N$. Therefore $\phi_n(t)$ converges, uniformly on bounded sets, to some nondecreasing bounded  function $\phi_*(t), \ \phi_*(0)=0.5$ (alternatively, we can use the Helly's selection theorem in order to get pointwise convergence of $\phi_n(t)$).
Taking limit,  as $n \to +\infty$, in an appropriate integral form of the differential equation
$$
\phi_n''(t) -c_n\phi_n'(t) + (B_n\phi_n)(t)(1 - (1-b)(S_nB_n\phi_n)(t)) =0,
$$
we find that $\phi_*(t)$ also satisfies (\ref{PEW}).  Since $\phi_*(0)= 0.5$, arguing as in the final part  of the proof of Theorem \ref{main1}, we conclude that
$\phi_*(+\infty)= 1, \ \phi_*(-\infty)=0$.
\subsection{Necessity of assumptions on $\chi_0(z)$}\label{SUB42}
The result announced  in the title of this subsection  follows from the next
\begin{lemma}\label{sms2}  If  problem (\ref{PEW}) has a monotone non-constant bounded solution $\phi(t)$, then  $\phi'(t)>0$ for all $t\in \R$  and  $\chi_0(z)$ has two positive zeros $\lambda_2 \leq \lambda_1$. Moreover,
for some $\varepsilon >0$, and an appropriate $t_a, d \in \R$, it holds
$$
\phi(t+t_a) = e^{\lambda_2t}((-t)^{j}+jd+ o(e^{\varepsilon t})), \ t \to -\infty,
$$
where $j=0$ if $\lambda_1 < \lambda_2$ and $j=1$ when $\lambda_1 =\lambda_2$.
 \end{lemma}
\begin{proof}  Let $z_1 <0<z_2$ denote the roots of equation $z^2-cz-1=0$. Consider
$$ N_1(t) = \alpha \left\{
\begin{array}{ccc}
e^{z_2 t} &,& t \leq 0 \\
e^{z_1 t}&,& t \geq 0
\end{array}, \right. $$
where positive $\alpha$ is chosen to comply with the normalization condition $\int_\R N_1(s)ds =1$.

Since bounded  profile $\phi$ satisfies the differential equation
\begin{equation} \label{pt}
\phi''(t) -c\phi'(t) -\phi(t) + \left(\phi(t)+(B\phi)(t)(1 - (1-b)(SB\phi)(t)\right) =0,
\end{equation}
we obtain that
\begin{equation}\label{ieq}
\phi(t) = \int_{-\infty}^{+\infty}N_1(t-s)\left[\phi(s)+(B\phi)(s)(1 - (1-b)(SB\phi)(s))\right]ds >0, \  t \in \R.
\end{equation}
Equation (\ref{pt}) and inequality in  (\ref{ieq}) imply  that $\phi'(t) =0,\ \phi''(t) \geq 0$ can not hold   simultaneously,  so that $\phi'(t) >0$ for all $t$.

  Next, take $N>0$ such that $q:=(2-b)\int_{-N}^NN_1(s)ds >1$.  Set $t_0:= \phi^{-1}(b)-N$, then
$0 <\phi(t) < b$ for all $t < t_0+N$, so that, using  monotonicity of $\phi(t)$,  we conclude that
\begin{eqnarray}
\phi(t) &\geq& \int_{t-N}^{t+N}N_1(t-s)[\phi(s)+(1-b)(B\phi)(s)]ds \n\\
&\geq&  (2-b)\int_{t-N}^{t+N}N_1(t-s)\phi(s)ds
  \n\\
&\geq&  q\phi(t-N),\n
   \end{eqnarray}
for all $t < t_0$.
Since $q>1$, the latter implies that $0\leq \phi(t) \leq de^{-\gamma t}, \ t \leq t_0,$ where $\gamma =N^{-1}\ln q >0$ and $d$ is some positive constant.
A lower exponential estimate for $\phi$ can also be obtained:
$$
\phi(t) \geq \int_{-\infty}^{+\infty}N_1(t-s)\phi(s)ds \geq \int_{0}^{1}N_1(t-s)\phi(s)ds \geq \alpha \phi(0) e^{z_2(t-1)}, \ t \leq 0.
$$
Since $\phi(t)$ is decaying exponentially but not super-exponentially, we can apply  \cite[Lemma 22]{tat} (see also \cite[Lemma 28]{GT}) to conclude that
$$
y(t) = -(1-b)\sum_k \mbox{Res}_{z=\lambda_k}\left[\frac{e^{zt}}{\chi_0(z)}\int_\R e^{-zs}(B\phi)(s)(SB\phi)(s)ds\right](1+o(e^{\epsilon t})), \ t \to -\infty.
$$
where the sum is calculated over some finite set of zeros $\lambda_k, \ \Re \lambda_k >0$ of $\chi_0(z)$ and $\epsilon$ is some small positive number.  After a straightforward calculation, invoking with  the positivity of $(B\phi)(t)(SB\phi)(t)$ and $\phi(t)$ on $\R$, we conclude that $\chi_0(z)$ must have positive zeros $\lambda_2 \leq \lambda_1$. This yields  the required asymptotic formula
$$
y(t) = -(1-b)\mbox{Res}_{z=\lambda_2}\left[\frac{e^{zt}}{\chi_0(z)}\int_\R e^{-zs}(B\phi)(s)(SB\phi)(s)ds\right](1+e^{\epsilon t}),\  t \to -\infty.  \quad \square
$$
\end{proof}

\subsection{Necessity of assumptions on $\chi_1(z)$}
Lemmas \ref{L4} and \ref{sms2} imply that every admissible propagation speed $c$ for equation (\ref{PE}) is bigger than $2$, $c >2$. Therefore the quadratic
equation $z^2-cz+1=0$ has exactly two positive roots, $r_1<r_2$ and, for every $f\in C_b(\R,\R)$, the unique bounded on $\R$ solution
of equation
$
y''-cy+y =f(t)
$
can be written in the form $y(t) = \int_t^{+\infty}K_2(t-s)f(s)ds$ where $K_2(s) = \beta(e^{r_1s}-e^{r_2s}) >0,\ s < 0$ and positive normalization constant
$\beta$ is chosen to satisfy $\int_{\R_-} K_2(s)ds=1$, cf. \cite{GT}.  We will use this observation and a nice trick proposed in \cite[p. 3050]{FZ} to prove
the following result.
\begin{lemma} \label{nk} If  $\chi_1(z)$ does not have negative zeros,  the inequality $y(t): = 1 - \phi(t) >0$ fails to hold for all  $t \in \R$.
\end{lemma}
\begin{proof} Assume that  $\chi_1(z)$ does not have negative zeros and, nevertheless, $y(t) >0$ for all $t\in \R$. By a straightforward calculation, we obtain that
$$
y''(t)-cy'(t) - (1-b)(B\phi)(t)(SBy)(t) =0,
$$
so that,
$$
y(t) = \int_t^{+\infty}K_2(t-s)\left[y(s) +  (1-b)(B\phi)(s)(SBy)(s)\right]ds.
$$
Since $(1-b)(B\phi)(t)$ is monotonically increasing to $1$,  for every small $\delta \in (0,0.5)$ there exists $t_b$ such that $(1-b)(B\phi)(t) > 1-\delta$ for $t \geq t_b$.
Hence,
\begin{equation}\label{yy}
y(t) \geq \int_t^{+\infty}K_2(t-s)\left[y(s) +  (1-\delta)(SBy)(s)\right]ds, \ t \geq t_b.
\end{equation}
In particular, setting $p = 0.5\int_{-c\tau/2}^0K_2(s)ds \in (0,1)$, we find that
$$
y(t) \geq (1-\delta)\int_t^{t+c\tau/2}K_2(t-s)y(s-c\tau)ds \geq py(t-c\tau/2) , \ t \geq t_b.
$$
Therefore, for some $C>0$ and $\gamma : = -2\ln p/(c\tau) >0$, it holds that $y(t) \geq Ce^{-\gamma t}, \ t \geq t_b$. Thus the following non-negative number $\gamma_*$ is well defined and satisfies
$$
- \gamma_* := \liminf_{t +\infty}\frac 1 t \ln y(t)      \geq -\gamma.
$$
Next, there are $\delta >0$ small enough and an integer $j_0$ large enough to satisfy
\begin{equation}\label{ga}
(1-\delta)e^{c\tau\gamma_*}\sum_{j=0}^{j_0} b^je^{jc\tau\gamma_*} > \gamma_*^2+c\gamma_*.
\end{equation}
This is obvious  when $be^{c\tau\gamma_*} \geq 1$. Now, if $be^{c\tau\gamma_*} <1$ then
$$e^{c\tau\gamma_*}\sum_{j=0}^{+\infty} b^je^{jc\tau\gamma_*} = \frac{e^{\gamma_* c\tau}}{1-be^{\gamma_* c\tau}} > \gamma_*^2+c\gamma_*,
$$
in view of our assumption on the function $\chi_1(z)$ evaluated in $z =-\gamma_*<0$.
This shows that (\ref{ga}) is true for sufficiently large $j_0$ and sufficiently small $\delta$.
Consequently, there exists $d>1$ such that
$$
(\gamma^2+c\gamma+1)^{-1} [1 +  (1-\delta)e^{c\tau\gamma} \sum_{j=0}^{j_0}b^j e^{jc\tau\gamma} ] >d >1
$$
for all $\gamma> \gamma_*$ sufficiently close to $\gamma_*$.  Suppose that, in addition,  such $\gamma$ satisfies also the inequality
$\gamma < \gamma_* + \ln d/(j_0c\tau)$.
Then using the estimate $y(t) \geq Ce^{-\gamma t}, \ t \geq t_b,$ in (\ref{yy})  for $t \geq t_b + j_0c\tau$, we find that
\begin{eqnarray}
y(t) &\geq&  C\int_t^{+\infty}K_2(t-s)e^{-\gamma s} [1 +  (1-\delta)e^{c\tau\gamma} \sum_{j=0}^{j_0}b^j e^{jc\tau\gamma} ]ds \n\\
&\geq&  C\frac{ [1 +  (1-\delta)e^{c\tau\gamma} \sum_{j=0}^{j_0}b^j e^{jc\tau\gamma} ]}{\gamma^2+c\gamma+1} e^{-\gamma s} \n\\ &>&  Cd e^{-\gamma t},\n
\end{eqnarray} for $ t \geq t_b + j_0c\tau.$
This shows that $y(t) \geq Cd^k e^{-\gamma t}, \ t \geq t_b+kj_0c\tau$, so that
$$
-\gamma_* = \liminf_{t +\infty}\frac 1 t \ln y(t)  \geq \lim_{t +\infty}\frac 1 t \ln (d^{t/(j_0c\tau)} e^{-\gamma t}) = -\gamma+ \ln d/(j_0c\tau) > -\gamma_*.
$$
This contradiction completes   the proof of Lemma \ref{nk}.  \hfill $\square$
\end{proof}

\section{Proof of Theorem \ref{main3}}
\subsection{Non-existence of super-exponentially decaying solutions at $-\infty$}
We will  need a non-local version of Lemma 3.6 in \cite{VT} (see also  \cite[Lemma 6]{SoT}). It will allow to
exclude  super-exponentially decaying (at $-\infty$) solutions to some asymptotically autonomous differential equations with infinite delay.
In sequel, for a small positive fixed $\rho$ we will define  the following Banach space of  fading memory type \cite{HK}:
$$
UC_\rho =\{g:\R_- \to \R^m \ \mbox{is such that} \ g(t)e^{\rho t} \ \mbox{is uniformly continuous and bounded on} \ \R_-\}.
$$
$UC_\rho$ will be  equipped with the norm $|g|_\rho:= \sup_{s \leq 0}| g(t)e^{\rho t}|$.

\begin{lemma} \label{sms3} Suppose that $L:UC_\rho \to \R^m$ is continuous linear operator
and $M: (-\infty, 0] \times UC_\rho \to \R^m$ is a continuous function such that
$|M(t,\phi)| \leq \mu(t)|\phi|_\rho$ for some non-negative $\mu(t) \to 0$ as $t\to -\infty$.
Then the system
\begin{equation} \label{LM}
x'(t) = Lx_t+M(t,x_t), \  x_t(s):= x(t+s), \ s \leq 0,
\end{equation}
does not have nontrivial exponentially small solutions at $-\infty$ (i.e. non-zero solutions $x:\R_- \to \R^m$ such that for each $\gamma \in \R$ it holds that $x(t)e^{\gamma t} \to 0, \ t \to -\infty$).
\end{lemma}
\begin{proof}  On the contrary, suppose that there exists a non-zero small solution $x(t)$ of $(\ref{LM})$ at $-\infty$. 
Set $|x_t|_C:= \sup\{|x(s)|; \ s \leq t\}$. 
Then  $|x_t|_Ce^{-\gamma t} \to 0, \ t \to -\infty$ for each  $\gamma >0$. 
Take some $\sigma>0$.   We claim that smallness of 
 $x(t)$ implies that 
 $
 \inf_{t \leq 0} |x_{t-\sigma}|_C/|x_t|_C =0. 
 $
 Indeed, otherwise there is $K\in (0,1)$ such that $|x_{t-\sigma}|_C/|x_t|_C \geq K, \ t \leq 0,$ and therefore, setting $\nu:= \sigma^{-1}\ln K$, we obtain the following contradiction: 
 $$
0<  |x_t|_Ce^{\nu t} \leq |x_{t-\sigma}|_Ce^{\nu (t-\sigma)}\leq |x_{t-2\sigma}|_Ce^{\nu (t-2\sigma)}\leq\dots\leq |x_{t-m\sigma}|_Ce^{\nu (t-m\sigma)} \to 0, \ m \to \infty.  
$$
Hence, there is  a sequence $t_j\to -\infty$ such that  $|x_{t_j-\sigma}|_C/|x_{t_j}|_C \to 0$ as $j \to \infty$. Clearly,  for all large $j$,
$|x_{t_j}|_{C} = |x(s_j)|$ for some $s_j \in [t_j-\sigma,t_j]$ and it holds $|x(s_j)| \geq |x(s)|, \ s \leq t_j$.   Since $0\leq t_j-s_j \leq \sigma$, 
without loss of generality we can assume that $\sigma_j:= t_j-s_j \to \sigma_* \in [0,\sigma]$. 

Now, for sufficiently large $j$, consider the sequence 
of functions $y_j \in UC_\rho, \ |y_j|_\rho \leq 1$: 
$$
y_j(t)= \frac{x(t+t_j)}{|x(s_j)|}, \ t \leq 0, \quad |y_j(-\sigma_j)| =1, \quad  |y_j(t)| \leq 1, \ t \leq 0. 
$$
For each $j$ and $t \leq 0$,  $
y_j(t)$ satisfies the equations 
$$
y'(t) = Ly_t+\frac{M(t+t_j,x_{t+t_j})}{|x(s_j)|}, \quad y_j(t) = y_j(-\sigma_j)+ \int^t_{-\sigma_j}\left(Ly_u+\frac{M(u+t_j,x_{u+t_j})}{|x(s_j)|}\right)du, 
$$
and therefore   
$
 |y_j(t)| \leq 1, \ |y_j'(t)| \leq \|L\|+\sup_{s \leq t_j} \mu(s) \leq \|L\|+\sup_{s \leq 0} \mu(s),$ $t\leq 0$, $j \in \N
$ (here $\|\cdot\|$ denotes the operator norm).
Thus, due to the Arzel\`a-Ascoli theorem,  there exists a subsequence $y_{j_k}(t)$ converging, uniformly on  compact subsets of $\R_-$, to some continuous function $y_*(t)$ such that 
$|y_*(-\sigma_*)|=1$ and $y_*(t)=0$ for all $t \leq -\sigma$. 
Clearly, for each fixed $t \leq 0$,  sequence $(y_{j_k})_t$ also converges to $(y_*)_t$ in $UC_\rho$,  and 
$$\frac{|M(t+t_{j_k},x_{t+t_{j_k}})|}{|x(s_{j_k})|}\leq \mu(t+t_{j_k}) \to 0, \quad k \to +\infty,$$
so that by the Lebesgue's bounded convergence theorem, 
$$
y_*(t) = y_*(-\sigma_*)+ \int^t_{-\sigma_*}L(y_*)_udu,\quad  t\leq 0. 
$$
In particular, $y'_*(t) = L(y_*)_t,\  t\leq 0$.  Since $y_*(t)= 0$ for all $t\leq -\sigma $,  the existence and uniqueness theorem applied to the initial 
value problem $y'(t) = Ly_t,$ $ t\leq -\sigma, \ y_\sigma =0,$ implies that  also $y_*(t)= 0$  for all $t \geq -\sigma$. However, this contradicts the fact  that
$|y_*(-\sigma_*)|=1$. The proof of Lemma \ref{sms3} is completed. 
\hfill $\square$
\end{proof}
\subsection{Proof of the wavefront uniqueness}
Suppose that $\psi(t)$, $\phi(t)$ are monotone solutions of the boundary value problem (\ref{PEW}). After shifting $\psi(t)$ if necessary,  we can assume that $\psi$ and $\phi$ have the same leading term in their asymptotic representations given in Lemma \ref{sms2}.  This implies that the function $y(t) = (\psi(t)-\phi(t))e^{-\lambda t}$ where $\lambda \in [\lambda_1(c), \lambda_2(c)]$ is bounded on $\R$ and $y(+\infty)=0$.

Next, consider operators $B_\lambda,  SB_\lambda$ defined on bounded functions $w:\R_-\to \R$
by $B_\lambda w = \sum_{j\geq 0}b^je^{-j\lambda c \tau}w(-jc\tau), \ SB_\lambda w = \sum_{j\geq 0}b^je^{-(j+1)\lambda c \tau}w(-(j+1)c\tau)$. For bounded functions $\psi:\R\to \R$,  it is easy to check  the commutativity relations
$$e^{-\lambda t}(B\psi)(t) = B_\lambda [\psi(\cdot)e^{-\lambda \cdot}]_t, \quad e^{-\lambda t}(SB\psi)(t) = SB_\lambda [\psi(\cdot)e^{-\lambda \cdot}]_t,$$
implying that $y(t)$ satisfies the functional  differential equation
\begin{equation}\label{seca}
y''(t) - (c-2\lambda)y'(t) + (\lambda^2-c\lambda)y(t) + B_\lambda y_t = N(t,y_t),  \ t \in \R,
\end{equation}
with
$$
N(t,y_t):= (B\phi)(t)SB_\lambda y_t+ (SB\psi)(t)B_\lambda y_t.
$$
A straightforward calculation shows that,  for $\rho \in (0, \lambda)$,
the norms of operators $SB_\lambda, \ B_\lambda: UC_\rho \to \R$ and $N: \R_-\times UC_\rho \to \R$ satisfy the following estimates:
$$
\max\{\|SB_\lambda\|, \|B_\lambda\|\} \leq \frac{1}{1-b}, \quad |N(t,w)| \leq   {|w|_\rho}\left((B\phi)(t) + (SB\psi)(t)\right), \ w \in UC_\rho.
$$
We start by analyzing the noncritical speeds $c >c_*(\tau)$. In such a case, $\lambda_2(c) < \lambda_1(c)$ so that we can choose $\lambda  \in (\lambda_2(c), \lambda_1(c))$.
Then the principal term of asymptotic representation of $y(t)$ at $-\infty$ has the form $Ae^{(\lambda_1(t)-\lambda)t}$.

\underline{Case I}. Suppose  that $A=0$.  Since the eigenvalues of  linear equation
$$
y''(t) - (c-2\lambda)y'(t) + (\lambda^2-c\lambda)y(t) + B_\lambda y_t = 0
$$
coincide with the zeros of $\chi_0(z-\lambda)$, Lemma \ref{L4} implies that $\lambda_1-\lambda$ is the unique eigenvalue of the above homogeneous equation in the half-plane
$\Re z \geq 0$. Furthermore, $N(t,y_t) = O(y(t)e^{\lambda_2t}), \ t \to -\infty$. Thus \cite[Lemma 22]{tat} implies that $y(t)$ decays super-exponentially  at $-\infty$.  Consequently,
it is easy to find that $y'(t)$ is also super-exponentially decaying at $-\infty$ (for instance, choose $\lambda >c/2$ and find $v(t)=y'(t)$ solving (\ref{seca}) on $(-\infty,t]$ as the first order ordinary differential equation with respect to $v$). Thus an application of Lemma \ref{sms3} gives that $y(t) \equiv 0$  that amounts to the front uniqueness.

\underline{Case II}. Hence, it suffices to consider  $A>0$ (if $A<0$, we can interchange roles of $\phi$ and $\psi$).  In such a case, $y(t)$ is positive on some maximal open interval
$\Sigma:= (-\infty, \sigma)$, and $y(-\infty)= y(\sigma) =0$ (the interval $\Sigma =\R$ is admitted).  Let $a \in \Sigma$ be the absolute maximum point of $y(t)$ on $\Sigma$. Then clearly
$N(a,y_a)>0$  and $$y''(a) - (c-2\lambda)y'(a) + (\lambda^2-c\lambda)y(a) + B_\lambda y_a < (\lambda^2-c\lambda)y(a) + B_\lambda (y(a))= \chi_0(\lambda)y(a) <0,$$
in contradiction with (\ref{seca}). Thus the case $A\not=0$ can not happen.

\vspace{2mm}

Finally, consider the critical speed $c =c_*(\tau)$. Then there exists finite limit $y(-\infty) =A$; without loss of generality, we can assume that
$A \geq 0$.  Now, if $A=0$, by the same reasons as presented in Case I above, $y(t)\equiv 0$ and the wave uniqueness follows.  If $A>0$, then  $y(t)$ is positive on some maximal open interval  $\Sigma:= (-\infty, \sigma)$, and $y(\sigma) =0$.  Again, after a partial integration of equation (\ref{seca}), we obtain that $v(t)=y'(t)$ is bounded at $-\infty$. Therefore
$y''(t)$ is also bounded at $-\infty$ so that $y'(t)$ is uniformly continuous on $\R_-$. Thus Barbalat's lemma \cite{wz} implies that $y'(-\infty)=0$.  Now, using relations $\chi_0(\lambda)= \chi_0'(\lambda)=0$ with $\lambda = \lambda_2(c)$, and integrating (\ref{seca}) on some interval  $[\alpha,\beta] \subset (-\infty, \sigma)$ we obtain that
$$
y'(\beta)-y'(\alpha) + (2\lambda-c)(y(\beta)-y(\alpha)) + \sum_{j \geq 0} b^je^{-\lambda j c\tau} \left[ \int_{\alpha-jc\tau}^\alpha y(s)ds-\int_{\beta-jc\tau}^\beta y(s)ds\right] = \int_\alpha^\beta Nds.
$$
Now letting $\alpha\to -\infty,\ \beta \to \sigma$, we find that
\begin{eqnarray}
0 < \int^\sigma_{-\infty}N(s,y_s)ds &=& y'(\sigma) -  (2\lambda-c - \sum_{j \geq 0} (jc\tau) b^je^{-\lambda j c\tau} )A \n\\&& -  \sum_{j \geq 0} b^je^{-\lambda j c\tau}\int_{\sigma -jc\tau}^\sigma y(s)ds\n\\
&=&y'(\sigma) -   \sum_{j \geq 0} b^je^{-\lambda j c\tau}\int_{\sigma-jc\tau}^\sigma y(s)ds <0.\n
\end{eqnarray}
This  contradiction shows that necessarily $A=0$ completing  the proof of Theorem \ref{main3}. \hfill $\square$
\section*{Appendix}
\subsection{On upper and lower solutions}
In Definition \ref{D1} of super- and sub-solutions, their first derivatives are allowed to have a jump discontinuity whenever the sign of jump is non-positive
(for super-solutions) or non-negative (for sub-solutions). Actually, even if the number of such discontinuities does not matter,  in practice it suffices to use only one point. Considering
$C^2$-smooth (on $\R$) functions  in  Definition \ref{D1}, we obtain the definition of the upper and lower solutions. Since in general it is rather difficult to find good initial approximations $\phi_\pm$ for the waves, the advantage  of working with less demanding super- and sub-solutions instead of  upper and lower solutions becomes quite evident.
The main result of this section is well known in several particular cases,  cf.  \cite[Lemma 2.5]{Ma}, \cite{wzE}, \cite[Lemma 15]{GT}. However, the proofs of it given in the mentioned  papers are not suitable to cope with the general situation described in Theorem \ref{main1}. Instead, here we are using the  mollification technique.
\begin{lemma} \label{L7}Let $C^\infty$-smooth function  $k_\delta: \R \to \R_+$ have compact support contained in  $[-\delta,\delta]$ and satisfy
$\int_Rk_\delta(s)ds=1$.  Suppose that $\phi_+, \phi_-$ are super- and sub-solutions given in Definition \ref{D1}. Then
$$\psi_\pm(t) = \int_{-\delta}^{\delta}\phi_\pm(t-s)k_\delta(s)ds, \quad t \in \R,$$
are $C^\infty$-smooth  functions such that
\begin{eqnarray}\label{5}
g_+(t,\delta):=D\psi_+''(t)  -c\psi'_+(t) +  \int_{-\delta}^{\delta}k_\delta(s)F(\phi_+)(t-s)ds \leq 0, \ t \in \R; \\ \label{6}
g_-(t,\delta):= D\psi_-''(t)  -c\psi'_-(t) +  \int_{-\delta}^{\delta}k_\delta(s)F(\phi_-)(t-s)ds \geq 0, \ t \in \R.
\end{eqnarray}
In addition, $\psi_\pm(t)$ possess  another properties of respective  $\phi_\pm(t)$  listed in Definition \ref{D1}. Furthermore, under assumptions of Theorem \ref{main1}, we have that $\phi_- \leq A\phi_-$, $A\phi_+ \leq \phi_+$,
where $A(\phi):= -\mathcal{I}[F(\phi)+L\phi]$.
\end{lemma}
\begin{proof} Clearly,
$$\psi_\pm(t) = \int_{t-\delta}^{t+\delta}\phi_\pm(s) k_\delta(t-s)ds$$
are $C^\infty$-smooth and have the same asymptotic, monotonicity and sign  properties as  $\phi_\pm(t)$. Next, for all $t \not \in (t_1-\delta, t_1+\delta)$,
it is immediate to calculate the first and the second derivatives
$$
\psi^{(j)}_\pm (t) = \int_{-\delta}^{\delta}\phi^{(j)}_\pm(t-s)k_\delta(s)ds, \ j =1,2.
$$
Now, if $t \in (t_1-\delta, t_1+\delta)$, we can use the representation
$$\psi_\pm(t) = \int_{t-\delta}^{t_1}\phi_\pm(s) k_\delta(t-s)ds+ \int_{t_1}^{t+\delta}\phi_\pm(s) k_\delta(t-s)ds$$
to find, after integrating by parts, that
$$
\psi'_\pm (t) = \int_{-\delta}^{\delta}\phi'_\pm(t-s)k_\delta(s)ds, $$
$$ \psi''_\pm (t) = \int_{-\delta}^{\delta}\phi''_\pm(t-s)k_\delta(s)ds + k_\delta(t-t_1)\left[\phi_\pm'(t_1+)-\phi_\pm'(t_1-) \right]
$$
(in the latter formula, we can set formally $\phi''_\pm(t_1) =0$).

Next,  after multiplying differential inequalities of Definition \ref{D1} evaluated in $t-s$ by $k_\delta(s)$ and integrating them between $-\delta$ and $+\delta$ with respect to $s$,
we easily obtain (\ref{5}) and (\ref{6}) in view of the assumed sign restrictions  on $\phi_\pm'(t_1+)-\phi_\pm'(t_1-)$.

Finally, we have that
 \begin{eqnarray}
\psi_+(t)& =& - \int_{\R}N(t-s) \left[-g_+(s,\delta) + L\psi_+(s)+  \int_{-\delta}^{\delta}k_\delta(s)F(\phi_+)(s-v)dv \right]ds \n\\
& \geq&
 \int_{\R}N(t-s) \left[ L\psi_+(s)+  \int_{-\delta}^{\delta}k_\delta(s)F(\phi_+)(s-v)dv \right]ds.\n
 \end{eqnarray}
By $p$-continuity of $L$ and the Lebesgue's dominated convergence theorem, after taking limit, as $\delta \to 0+$, in the latter inequality,
we find that $\phi_+ \geq A\phi_+$. The proof of inequality  $\phi_- \leq A\phi_-$  is similar and therefore it is omitted. \hfill $\square$
\end{proof}
\subsection{Curves defined by the characteristic equations}
\begin{lemma} \label{LL} Assume that $b \in (0,1)$ is a fixed number, $c, \tau > 0$. Then there exist finite $\tau(b)>0$ and a strictly decreasing continuous function $c_\#=c_\#(\tau) >0,\ \tau >\tau(b),$ satisfying
$
c_\#(\tau(b)-)=+\infty, \  c_\#(+\infty) =0
$
and such that  $\chi_1(z)$ has exactly two negative zeros $\mu_2(c) \leq \mu_1(c)$ (counting multiplicity) if and only if
either $\tau \in (0,\tau(b)]$ or $\tau > \tau(b)$ and
$c \leq c_\#(\tau)$. These zeros coincide only if $c=c_\#(\tau)$.  Function $\tau(b)$ is strictly decreasing on $(0,1)$, with $\tau(0+) = 1/e$ and $\tau(1-)= 0$.

\end{lemma}
\begin{proof} It is easy to find that $\chi_1'''(x)\not=0$ for all real $x \not= \ln b/(c\tau)$ and that $\chi_1(x)$ has exactly one non-negative real zero (which is simple, in addition). Thus $\chi_1(z)$ can have at most two negative zeros.  After introducing the change of variable $\lambda= cz, \ \epsilon = c^{-2}$, equation $\chi_1(z) =0$ takes the form
\begin{equation}\label{EE}
\epsilon z^2 -z = \frac{1}{e^{z\tau}-b}.
\end{equation}
It shows clearly that for each $\tau > 0$ and  $b\in (0,1)$ there exists a unique non-negative real number $\epsilon_\#$ such that (\ref{EE}) has negative solutions if and only
$\epsilon \geq \epsilon_\#$.  Thus $\epsilon_\# =0$ if and only if equation
\begin{equation}\label{EE2}
-z = \frac{1}{e^{z\tau}-b}, \ z <0.
\end{equation}
has negative solutions. Since the right-hand side of (\ref{EE2}) is an increasing function of $\tau$, we deduce that there exists some positive $\tau(b)$ such that (\ref{EE2})
has negative roots for all $\tau \in (0,\tau(b)]$ and does not have real roots if $\tau > \tau(b)$.  In other words, $\epsilon_\#=0$ if and only if $\tau \in (0,\tau(b)]$. Actually, the above arguments imply all above mentioned properties of $c_\#(\tau)$, see also (\ref{D12}) below. These arguments also shows that $\tau =\tau(b)$ can be determined as a unique postive number for which equation (\ref{EE2}) has a negative double root $z_0$. Thus
\begin{equation}\label{EE3}
-z_0 = \frac{1}{e^{z_0\tau}-b}, \quad -1 = \frac{-\tau e^{z_0\tau}}{(e^{z_0\tau}-b)^2},
\end{equation}
from which
$$
1=z_0^2\tau e^{z_0\tau},\,\,\, z_0 b -\frac{1}{z_0\tau} =1, \quad \mbox{so that,} \quad z_0\tau = -2(1+\sqrt{1+4b/\tau})^{-1}=:-\sigma, \ \sigma \in (0,1).
$$
Hence, we have obtained a  parametric solution of  (\ref{EE3}):
$\tau = \sigma^2e^{-\sigma}, \ b = (1-\sigma)e^{-\sigma},$ $ \sigma \in (0,1)$. Note that $\tau'(b) = -\sigma<0$, other properties of $\tau(b)$ are also obvious.
 \hfill $\square$
\end{proof}

\begin{lemma} \label{IT} The curves $c=c_*(\tau)$ and $c=c_\#(\tau)$, $\tau > \tau(b)$,  have exactly one intersection point $(\tau_0, c_0)$ where
$c_*'(\tau_0) >c_\#'(\tau_0)$.
\end{lemma}
\begin{proof}For $c=c_*(\tau)$ [respectively, $c=c_\#(\tau)$] the characteristic  function $\chi_0(z)$ [respectively, $\chi_1(z)$] has positive multiple zeros  $\lambda_2(\tau) = \lambda_1(\tau)$  [respectively, negative zeros $\mu_1(\tau)=\mu_2(\tau)$]. After introducing change of variables $\epsilon = c^{-2}(\tau), \ z = \lambda_2(\tau)c(\tau)$,  we observe that
$\epsilon, z,\tau$ satisfy the system
$$
F(z,\epsilon,\tau) =0, \quad F_z(z,\epsilon,\tau) =0, \quad \mbox{where} \quad  F(z,\epsilon,\tau) = \epsilon z^2-z + \frac{1}{1-be^{-z\tau}}.
$$
Thus, after differentiating the above equations with respect to $\tau$, we can find that
\begin{equation}\label{D11}
c_*'(\tau) = -\frac{c_*(\tau)}{\tau} + \frac{c^2_*(\tau)}{2\lambda_2(\tau)\tau} <0, \ \tau >0.
\end{equation}
Here we use that obvious fact that $\lambda_2(\tau) > c_*(\tau)/2$ for $\tau >0$.  Similarly,
\begin{equation} \label{D12}
c_\#'(\tau) = -\frac{c_\#(\tau)}{\tau} + \frac{c^2_\#(\tau)}{2\mu_2(\tau)\tau} <0.
\end{equation}
(Curiously, these differential relations coincide with equations in \cite[A.3: proof of Lemma 1.3, p. 66]{LMS} derived in a different situation). 
It is obvious that due to  the asymptotic properties of $c=c_*(\tau)$ and $c=c_\#(\tau)$ described in Lemmas \ref{L4} and \ref{L5}, the graphs of these functions have at least one intersection at some point  $(\tau_0, c_0)$.
Since $\mu_2(\tau_0) < 0 < \lambda_2(\tau_0)$, differential relations (\ref{D11}) and (\ref{D12}) yield
$c_*'(\tau_0) >c_\#'(\tau_0)$.  This implies the uniqueness of the intersection point $\tau_0$. \hfill $\square$
\end{proof}
\section*{Acknowledgments}  \noindent
This work was initiated during a research stay of S.T. at the S\~ao Paulo University at  Ribeir\~ao Preto,
 Brasil. It was  supported by  FAPESP (Brasil) project 18/06658-1 and partially by  FONDECYT  (Chile) project 1190712.  S.T.  acknowledges the very kind hospitality of the  DCM-USP  and expresses his sincere gratitude to the Professors  \mbox{M. Pierri and E. Hernández}  for their  support   and   hospitality.


\end{document}